\documentclass[11pt]{amsart}

\usepackage[margin=0.95in]{geometry}
\usepackage{amsmath, color}
\usepackage{amssymb}
\usepackage{amsthm}
\usepackage{algpseudocode}
\usepackage{enumitem}
\usepackage{mdwlist}
\usepackage{tikz}
\usepackage{tikz-cd}

\theoremstyle{plain}
\newtheorem{theorem}{Theorem}[section]

\newtheorem{lemma}[theorem]{Lemma}
\newtheorem{corollary}[theorem]{Corollary}
\newtheorem{problem}[theorem]{Problem}
\newtheorem{conjecture}[theorem]{Conjecture}

\theoremstyle{definition}

\newtheorem{remark}[theorem]{Remark}

\DeclareMathOperator{\Ast}{Ast}
\DeclareMathOperator{\Ker}{Ker}
\DeclareMathOperator{\Img}{Im}
\DeclareMathOperator{\Lk}{Lk}
\DeclareMathOperator{\St}{St}

\DeclareMathOperator{\Facet}{Facet}

\newenvironment{enum*}
{\parskip=-4pt
\begin{enumerate*}

}
{\end{enumerate*}}

\begin{document}
\title{Flag complexes and homology}
\author{Kai Fong Ernest Chong}
\address{Singapore University of Technology and Design\\
      Singapore}
\email{ernest\_chong@sutd.edu.sg}

\author{Eran Nevo}
\address{Einstein Institute of Mathematics, Hebrew University of Jerusalem\\ Israel}
\email{nevo@math.huji.ac.il}
\thanks{Both authors were partially supported by the ISF--NRF Israel--Singapore joint research program, grant 2528/16.
E. Nevo also acknowledges support by the Israel Science Foundation grant ISF-1695/15 and by ISF--BSF joint grant 2016288.}


\begin{abstract}
We prove several relations on the $f$-vectors and Betti numbers of flag complexes. For every flag complex $\Delta$, we show that there exists a balanced complex with the same $f$-vector as $\Delta$, and whose top-dimensional Betti number is at least that of $\Delta$, thereby extending a theorem of Frohmader by additionally taking homology into consideration. We obtain upper bounds on the top-dimensional Betti number of $\Delta$ in terms of its face numbers. We also give a quantitative refinement of a theorem of Meshulam by establishing lower bounds on the $f$-vector of $\Delta$, in terms of the top-dimensional Betti number of $\Delta$. This result has a continuous analog: If $\Delta$ is a $(d-1)$-dimensional flag complex whose $(d-1)$-th reduced homology group has dimension $a\geq 0$ (over some field), then the $f$-polynomial of $\Delta$ satisfies the coefficient-wise inequality $f_{\Delta}(x) \geq (1 + (\sqrt[d]{a}+1)x)^d$.
\end{abstract}

\maketitle

\section{Introduction and Overview}\label{sec:Intro}
Flag complexes are (abstract) simplicial complexes $\Delta$ satisfying the property that every set of vertices of $\Delta$ that are pairwise connected by edges forms a face of $\Delta$. Barycentric subdivisions of polytopes, order complexes of posets, and Whitney triangulations of closed $2$-manifolds are some well-known examples of flag complexes. In graph theory, flag complexes are commonly known as clique complexes or independence complexes, since a flag complex can also be defined as the simplicial complex formed by the cliques in a graph, or equivalently by the independent sets in the corresponding complement graph. Consequently, some of the earliest results in extremal graph theory, such as Zykov's generalization~\cite{Zykov1949} of the celebrated Tur\'{a}n's theorem~\cite{Turan1941} in the 1940s, and works by Erd\H{o}s and his collaborators in the 1960s, e.g.~\cite{Erdos1962:NumberOfCompleteSubgraphsInCertainGraphs}, already give non-trivial bounds on the face numbers of flag complexes.

More recently, Frohmader~\cite{Frohmader2008:FaceVectorsOfFlagComplexes} proved that the $f$-vector of every flag complex is the $f$-vector of a balanced complex, thereby verifying a conjecture made independently by Kalai~\cite[p. 100]{book:StanleyGreenBook} and
Eckhoff~\cite{Eckhoff2004:NewTuranTypeTheoremForCliquesInGraphs}.
Recall that the \textit{$f$-vector} of a $(d-1)$-dimensional simplicial complex $\Delta$ is the vector $f(\Delta) = (f_{-1}(\Delta), f_0(\Delta), \dots, f_{d-1}(\Delta))$, where each $f_i(\Delta)$ is the number of $i$-dimensional faces in $\Delta$. (Note that $f_{-1}(\Delta) = 1$ for the \textit{empty face} $\emptyset$.) We say that $\Delta$ is \textit{balanced} if its vertices can be colored in $d$ colors such that every face has vertices of distinct colors. Let $\Bbbk$ be any fixed field, let $\widetilde{H}_k(\Delta) = \widetilde{H}_k(\Delta;\Bbbk)$ be the $k$-th reduced homology group of $\Delta$ with coefficients in $\Bbbk$, and let $\beta_k(\Delta) := \dim_{\Bbbk}(\widetilde{H}_k(\Delta))$ denote the $k$-th reduced Betti number of $\Delta$. We strengthen Frohmader's theorem by taking the top-dimensional homology of flag complexes into consideration:

\begin{theorem}\label{thm:FrohmaderWithHomology}
If $\Delta$ is a $d$-dimensional flag complex, then there exists a balanced complex $\Gamma$ with the same $f$-vector as $\Delta$, such that $\beta_d(\Gamma) \geq \beta_d(\Delta)$.
\end{theorem}

The Frankl--F\"{u}redi--Kalai theorem~\cite{FFK1988} gives an explicit numerical characterization of the $f$-vectors of balanced complexes, which can be stated in terms of what are called canonical representations. Thus by Frohmader's theorem~\cite{Frohmader2008:FaceVectorsOfFlagComplexes}, the $f$-vectors of flag complexes must satisfy certain inequalities involving canonical representations; see Section \ref{subsec:TuranComplexes} for details. To define canonical representations, we first need to introduce Tur\'{a}n graphs. A \textit{Tur\'{a}n graph} $T_d(n)$  is the complete $d$-partite graph on $n$ vertices, such that each of the $d$ partition sets has either $\lfloor \frac{n}{d}\rfloor$ or $\lceil \frac{n}{d} \rceil$ vertices. Note that if $n\leq d$, then $T_d(n)$ is the complete graph on $n$ vertices. The clique complex of $T_d(n)$, which we denote by $\Delta(T_d(n))$, is called a \textit{Tur\'{a}n complex}. For each $0\leq k\leq n$, let $\binom{n}{k}_{\!d}$ denote the number of $k$-cliques in $T_d(n)$. In particular, $\binom{n}{0}_{\!d} = 1$ for all $n\geq 0$. For convenience, we extend the definition of $\binom{n}{k}_{\!d}$ by allowing $n,k$ to be arbitrary integers: Define $\binom{n}{k}_{\!d} = 0$ if $k>n$, or $k<0$, or $n<0$.

For all positive integers $N, k, r$ satisfying $r\geq k$, there exist unique integers $s, N_k, N_{k-1}, \dots, N_{k-s}$ such that
\begin{equation}\label{eqn:CanonicalRepDefn}
N = \binom{N_k}{k}_{\!r} + \binom{N_{k-1}}{k-1}_{\!r-1} + \dots + \binom{N_{k-s}}{k-s}_{\!r-s},
\end{equation}
$N_{k-i} - \lfloor\frac{N_{k-i}}{r-i}\rfloor > N_{k-i-1}$ for all $0\leq i < s$, and $N_{k-s} \geq k-s > 0$. Such an expression in \eqref{eqn:CanonicalRepDefn} is called the \textit{$(k,r)$-canonical representation} of $N$; see \cite[Chap. 18]{book:HandbookDCGGoodmanRourke2004} (cf. \cite[Sec. 7]{Chong2015:HilbertFunctionsColoredQuotRings}, \cite[Lem. 3.6]{Frohmader2008:FaceVectorsOfFlagComplexes}).

Using Theorem \ref{thm:FrohmaderWithHomology} we provide an upper bound on the top-dimensional reduced Betti numbers of flag complexes, in terms of canonical representations:

\begin{theorem}\label{thm:FlagTopBettiUpperBound}
Let $\Delta$ be a $(d-1)$-dimensional flag complex such that $f_{k-1}(\Delta) = N > 0$. If $N$ has the $(k,d)$-canonical representation
\[N = \binom{N_{d}}{k}_{\!d} + \binom{N_{d-1}}{k-1}_{\!d-1} + \dots + \binom{N_{d-s}}{k-s}_{\!d-s},\]
then
\begin{equation}\label{eqn:UpperBoundTopBettiFlagComplex}
\beta_{d-1}(\Delta) \leq \binom{N_{d}-d}{d}_{\!d} + \binom{N_{d-1}-(d-1)}{d-1}_{\!d-1} + \dots + \binom{N_{d-s} - (d-s)}{d-s}_{\!d-s}.
\end{equation}
\end{theorem}

Meshulam~\cite[Thm. 1.1]{Meshulam2003:DominationHomology} proved that if $\Delta$ is a flag complex such that $\beta_{k-1}(\Delta) \neq 0$, then the $f$-vector of $\Delta$ satisfies $f_{i-1}(\Delta) \geq 2^i\binom{k}{i} = \binom{2k}{i}_{\!k}$ for all $i\geq 0$. As a consequence of Theorem~\ref{thm:FlagTopBettiUpperBound}, we give a quantitative refinement of Meshulam's theorem in the top-dimensional homology case, i.e. when $k-1=\dim \Delta$:

\begin{theorem}\label{thm:FlagLowerBoundfVecGivenTopBetti}
Let $\Delta$ be a $(d-1)$-dimensional flag complex with $\beta_{d-1}(\Delta) = a > 0$. If
\begin{equation}\label{eqn:TopBettiFlagCanonicalRep}
a = \binom{a_d}{d}_{\!d} + \binom{a_{d-1}}{d-1}_{\!d-1} + \dots + \binom{a_{d-s}}{d-s}_{\!d-s}
\end{equation}
is the $(d,d)$-canonical representation of $a$, then
\begin{equation}\label{eqn:fVecFlagLowerBound}
f_{i-1}(\Delta) \geq \binom{a_d+d}{i}_{\!d} + \binom{a_{d-1}+d-1}{i-1}_{\!d-1} + \dots + \binom{a_{d-s}+d-s}{i-s}_{\!d-s}
\end{equation}
for all $i\geq 0$. Furthermore, if equality holds in \eqref{eqn:fVecFlagLowerBound} for some $i=k$ satisfying $k\geq s+1$, then equality must hold in \eqref{eqn:fVecFlagLowerBound} for all $i\geq k$.
\end{theorem}

This gives the following corollary:

\begin{corollary}\label{cor:FlagLowerBoundfVecTuranGivenTopBetti}
Let $\Delta$ be a $(d-1)$-dimensional flag complex with $\beta_{d-1}(\Delta) = a > 0$. If $T$ is a $(d-1)$-dimensional Tur\'{a}n complex that satisfies $\beta_{d-1}(T) \leq a$, then $f_{i}(\Delta) \geq f_{i}(T)$ for all $i\geq 0$. Furthermore, if $\beta_{d-1}(T) = a$, then the following are equivalent: (i) $f_0(\Delta) = f_0(T)$; (ii) $f(\Delta) = f(T)$; and (iii) $\Delta \cong T$.
\end{corollary}

The following ``continuous'' analog of Theorem \ref{thm:FlagLowerBoundfVecGivenTopBetti} follows, which we state in terms of $f$-polynomials. For any $(d-1)$-dimensional simplicial complex $\Delta$, we define the \textit{$f$-polynomial} of $\Delta$ (in variable $x$) to be the polynomial
\[f_{\Delta}(x) := \sum_{i=0}^{d} f_{i-1}(\Delta)x^i.\]

\begin{theorem}\label{thm:FlagLowerBoundfPolyGivenTopBetti}
Let $\Delta$ be a $(d-1)$-dimensional flag complex with $\beta_{d-1}(\Delta) = a \geq 0$. Then the $f$-polynomial of $\Delta$ satisfies
$f_{\Delta}(x) \geq (1 + (\sqrt[d]{a}+1)x)^d$, where $\geq$ here denotes a coefficient-wise inequality of real polynomials. Furthermore, if $\sqrt[d]{a}$ is an integer, then this inequality is tight, with equality holding if and only if $\Delta \cong \Delta(T_d(d(\sqrt[d]{a}+1)))$.
\end{theorem}

\textbf{Outline:}
Section \ref{sec:Preliminaries} reviews basic definitions and notation, as well as collects useful results on Tur\'{a}n complexes and canonical representations. Section \ref{sec:HomologyColorShiftedBalanced} deals with the homology of color-shifted balanced complexes. This includes a balanced analog of Theorem \ref{thm:FlagTopBettiUpperBound}: We show that ``reverse-lexicographic'' balanced complexes maximize the top-dimensional reduced Betti numbers among all balanced complexes with the same number of top-dimensional faces (Theorem \ref{thm:RevlexComplexMaximizesBettiNum}). Theorems \ref{thm:FrohmaderWithHomology} and \ref{thm:FlagTopBettiUpperBound} are proven in Section \ref{sec:HomologyFlagComplexes}. In Section \ref{sec:LowerBoundsFaceNum}, we prove Theorem \ref{thm:FlagLowerBoundfVecGivenTopBetti}, Corollary \ref{cor:FlagLowerBoundfVecTuranGivenTopBetti}, and Theorem \ref{thm:FlagLowerBoundfPolyGivenTopBetti}, as well as their balanced analogs (Corollaries \ref{cor:BalancedfVecLowerBound} and \ref{cor:BalancedfPolyLowerBound}). Finally, we conclude in Section \ref{sec:FurtherRemarks} with related open problems.

\section{Preliminaries}\label{sec:Preliminaries}
Let $\mathbb{N}$ be the set of non-negative integers, and let $\mathbb{P}$ be the set of positive integers. For each $n\in \mathbb{P}$, define $[n] := \{1, \dots, n\}$. Throughout, let $U$ be a countable set, which we shall call the \textit{ground set}.

\subsection{Simplicial complexes}\label{subsec:SimplicialComplexes}
A \textit{simplicial complex} $\Delta$ on $U$ is a collection of subsets of $U$ that is closed under inclusion. Note that we do not require all singletons $\{u\}$ (for $u\in U$) to be contained in $\Delta$. Elements of $\Delta$ are called \textit{faces}, subsets of $U$ not in $\Delta$ are called \textit{non-faces}, and we assume that every simplicial complex is both finite and non-empty. The \textit{dimension} of each face $F\in \Delta$ is $\dim F := |F| - 1$, and the dimension of $\Delta$, denoted by $\dim \Delta$, is the maximum dimension of the faces. Maximal faces are called \textit{facets}, and $0$- (resp. $1$-)dimensional faces are called \textit{vertices} (resp. \textit{edges}). The collection of vertices and edges of $\Delta$ is called the \textit{underlying graph} of $\Delta$. If all facets of $\Delta$ have the same dimension, then $\Delta$ is \textit{pure}. For brevity, we say $\Delta$ is a ``$d$-complex'' to mean that $\Delta$ is a ``$d$-dimensional complex'', and we say $F$ is a ``$k$-face'' to mean that $F$ is a ``$k$-dimensional face''.

For each integer $k\geq -1$, let $\mathcal{F}_k(\Delta)$ be the set of all $k$-faces in $\Delta$, and recall that $f_k(\Delta)$ is the size of $\mathcal{F}_k(\Delta)$. In particular, $f_{-1}(\Delta) = 1$, since $\mathcal{F}_{-1}(\Delta)$ contains the \textit{empty face} $\emptyset$. If $\dim \Delta = d-1$, then the \textit{$h$-vector} of $\Delta$ is the vector $h(\Delta) = (h_0(\Delta), \dots, h_d(\Delta))$ that is uniquely determined by the equation
\begin{equation}\label{eqn:hPolynomial}
\sum_{i=0}^d h_i(\Delta)x^i = \sum_{i=0}^d f_{i-1}(\Delta)x^i(1-x)^{d-i}.
\end{equation}
The \textit{$h$-polynomial} of $\Delta$ (in variable $x$), which we denote by $h_{\Delta}(x)$, is defined to be the polynomial given in either side of \eqref{eqn:hPolynomial}.

Denote the \textit{vertex set} of $\Delta$ by $\mathcal{V}(\Delta) := \mathcal{F}_0(\Delta)$. Notice that each $v\in \mathcal{V}(\Delta)$ is of the form $v = \{u\}$ for some $u\in U$. We shall also define the \textit{vertex set} of each face $F\in \Delta$ to be the uniquely determined subset $\mathcal{V}(F)$ of $\mathcal{V}(\Delta)$ consisting of all vertices $v$ of $\Delta$ satisfying $v \subseteq F$. Also, let the set of facets of $\Delta$ be denoted by $\Facet(\Delta)$.

Given any finite collection $\mathcal{F} = \{F_1, \dots, F_k\}$ of subsets of $U$, there is a unique minimal simplicial complex that contains all elements in $\mathcal{F}$. Such a simplicial complex is said to be \textit{generated by} $\mathcal{F}$, and we shall denote this complex by $\langle \mathcal{F} \rangle$. For convenience, we shall write $\langle \{F_1, \dots, F_k\}\rangle$ simply as $\langle F_1, \dots, F_k\rangle$. If a simplicial complex $\Delta$ can be written as $\Delta = \langle F\rangle$ for a single subset $F \subseteq U$, then we say $\Delta$ is a \textit{simplex}.

A \textit{subcomplex} of $\Delta$ is a subcollection of $\Delta$ that is also a simplicial complex. Given any subset $W \subseteq U$, the subcomplex of $\Delta$ \textit{induced by $W$} is the simplicial complex $\Delta[W] := \{F \cap W: F\in \Delta\}$. Given any face $F\in \Delta$, the \textit{anti-star} of $F$ in $\Delta$ is the subcomplex $\Ast_{\Delta}(F) := \{G \in \Delta: G\cap F = \emptyset\}$, the \textit{link} of $F$ in $\Delta$ is the subcomplex $\Lk_{\Delta}(F) := \{G \in \Delta: G \cap F = \emptyset \text{ and }G \cup F \in \Delta\}$, and the \textit{open star} of $F$ in $\Delta$ is the collection of faces $\St_{\Delta}(F) := \{G\in \Delta: F\subseteq G\}$. The \textit{closed star} of $F$ in $\Delta$ is the subcomplex $\overline{\St}_{\Delta}(F) := \langle \St_{\Delta}(F)\rangle$. Notice that $\Lk_{\Delta}(F) = \Ast_{\Delta}(F) \cap \overline{\St}_{\Delta}(F)$. If $\Gamma_1$ and $\Gamma_2$ are simplicial complexes with disjoint vertex sets, then we define the \textit{join} of $\Gamma_1$ and $\Gamma_2$ to be the simplicial complex $\Gamma_1 * \Gamma_2 := \{F_1 \cup F_2: F_1 \in \Gamma_1, F_2 \in \Gamma_2\}$. If $v = \{u\}$ for some $u\in U$, such that $v\not\in \mathcal{V}(\Delta)$, then we say that $\Delta * \langle v\rangle$ is the \textit{cone} on $\Delta$ with \textit{conepoint} $v$.

A simplicial complex $\Delta$ is called \textit{Cohen--Macaulay} (over $\Bbbk$) if $\beta_i(\Lk_{\Delta}(F)) = 0$ for all $F\in \Delta$ and all $0\leq i < \dim(\Lk_{\Delta}(F))$. For a good introduction to Cohen--Macaulay complexes, see \cite{book:StanleyGreenBook}.

\subsection{Clique complexes of graphs}\label{sec:CliqueComplexesGraphs}
A \textit{graph} is an ordered pair $G = (V, E)$, such that $V$ is a finite set, and $E \subseteq \binom{V}{2}$. Elements of $V$ and $E$ are called {\it vertices} and {\it edges} respectively. If $E = \binom{V}{2}$, then we say that $G$ is {\it complete}. The {\it complement} of $G$ is the graph $\overline{G} := (V, \binom{V}{2}\backslash E)$. A {\it subgraph} of $G$ is a graph $G^{\prime} = (V^{\prime}, E^{\prime})$ such that $V^{\prime} \subseteq V$ and $E^{\prime} \subseteq E$. A {\it clique} of $G$ is a complete subgraph of $G$, and a {\it $k$-clique} of $G$ is a clique of $G$ with $k$ vertices.

The {\it clique complex} of $G$, denoted by $\Delta(G)$, is the simplicial complex formed by the cliques in $G$, and the {\it independence complex} of $G$ is the clique complex $\Delta(\overline{G})$. Notice that a {\it flag complex} can equivalently be defined as a simplical complex for which every minimal non-face has at most two elements. (If $\{u\} \in \mathcal{V}(\Delta)$ for every $u$ in the ground set $U$, then every minimal non-face of $\Delta$ has exactly two elements.) Here, the minimality is with respect to set inclusion. Note that every clique complex is a flag complex, and every flag complex is the clique complex of its underlying graph (treated as a graph).

\subsection{Colored complexes and balanced complexes}\label{subsec:ColoredComplexes}
A \textit{$d$-colored complex} is a pair $(\Gamma,\pi)$, where $\Gamma$ is a simplicial complex on a ground set $U$, and $\pi = (U_1, \dots, U_d)$ is an ordered partition of $U$, such that every face $F$ of $\Gamma$ satisfies $|F \cap U_i| \leq 1$ for all $i\in [d]$. If in addition, we have $\dim \Gamma = d-1$, then we say that $(\Gamma, \pi)$ is a \textit{balanced complex}. A simplicial complex $\Delta$ on a ground set $U$ is called a \textit{$d$-colorable complex} if there exists an ordered partition $\pi = (U_1, \dots, U_d)$ of $U$ such that $(\Delta,\pi)$ is a $d$-colored complex. For convenience, we say that $\Gamma$ (resp. $\Delta$) is $d$-colored (resp. $d$-colorable) with respect to $\pi$. By abuse of notation, we say that $\Delta$ is a \textit{balanced complex} if there exists an ordered partition $\pi$ such that $(\Delta,\pi)$ is a balanced complex.

Suppose $(\Delta, \pi)$ is a $d$-colored complex, where $\pi = (U_1, \dots, U_d)$. We shall assume that the elements of each $U_i$ are labeled as $u_{i,j}$, i.e. $U_i = \{u_{i,j}\}_{j\in [\lambda_j]}$ for some $\lambda_j \in \mathbb{P}$ (if $U_i$ is finite), or $U_i = \{u_{i,j}\}_{j\in \mathbb{P}}$ (if $U_i$ is countably infinite). Such a labeling gives a natural linear order $u_{i,1} < u_{i,2} < u_{i,3} < \dots$ on each $U_i$. We say that $(\Delta, \pi)$ is \textit{color-shifted} if every $F\in \Delta$ and $i\in [d]$ satisfy the following property:
\[\text{$u_{i,j} \in F\text{ and }u_{i,j^{\prime}} \not\in F\text{ for some integers }1\leq j^{\prime} < j \Rightarrow (F\backslash\{u_{i,j}\}) \cup \{u_{i,j^{\prime}}\} \in \Delta$.}\]

For each $k\in \mathbb{N}$, let $\binom{\pi}{k}$ be the collection consisting of every subset $U' \subseteq U$ of size $k$ that satisfies $|U' \cap U_i| \leq 1$ for all $i\in [d]$. By definition, the $d$-colored complex $(\Delta,\pi)$ satisfies $\mathcal{F}_{k-1}(\Delta) \subseteq \binom{\pi}{k}$ for all $k\in [d]$. Next, consider the linear order on $U$ uniquely determined by $u_{i,j} > u_{i', j'}$ if $j>j'$; or $j=j'$ and $i>i'$. With this linear order, define the \textit{revlex} (reverse-lexicographic) order $\leq_{r\ell}$ on $\binom{\pi}{k}$ by $A >_{r\ell} B$ if and only if $\max(A-B) > \max(B-A)$. We then say that $(\Delta,\pi)$ is \textit{revlex} (or equivalently, that $\Delta$ is \textit{revlex} with respect to $\pi$) if for every $k\in [d]$, the following property holds:
\[\text{If $F, F' \in \binom{\pi}{k}$ such that $F >_{r\ell} F'$ and $F\in \Delta$, then $F' \in \Delta$.}\]

By abuse of notation, a $d$-colorable complex $\Delta$ is called \textit{color-shifted} (resp. \textit{revlex}) if there exists an ordered partition $\pi = (U_1, \dots, U_d)$ such that $(\Delta,\pi)$ is color-shifted (resp. revlex), where each $U_i$ (resp. $U$) is implicitly assumed to be linearly ordered. Frequently, we consider the ordered partition $\Pi_d = (\Pi_{d,1}, \dots, \Pi_{d,d})$, defined by $\Pi_{d,i} := \{k\in \mathbb{P}: k \equiv i\pmod{d}\}$ for each $i\in [d]$.

\subsection{Tur\'{a}n complexes and canonical representations}\label{subsec:TuranComplexes}
Let $n$ and $d$ be positive integers. By construction, the Tur\'{a}n complex $\Delta(T_d(n))$ is a pure balanced flag complex. For $0\leq k\leq n$, recall that $\binom{n}{k}_{\!d} = f_{k-1}(\Delta(T_d(n)))$.

\begin{lemma}\label{prop:HVectorTuranFVectorTuran}
If $n\ge d$, then the $h$-vector of $\Delta(T_d(n))$ equals the $f$-vector of $\Delta(T_d(n-d))$.
\end{lemma}

\begin{proof}
Recall that for arbitrary simplicial complexes $\Delta_1, \dots, \Delta_k$ with disjoint vertex sets, we have
\[h_{\Delta_1 * \dots * \Delta_k}(x) = \prod_{i\in [k]} h_{\Delta_i}(x);\]
this identity follows easily from the definitions of joins and $h$-polynomials.
The $h$-polynomial of a simplicial complex generated by $m>0$ vertices is $1+(m-1)x$, which coincides with the $f$-polynomial of a simplicial complex generated by $m-1$ vertices.
Thus,
\[h_{\Delta(T_d(n))}(x) = \prod_{i\in [d]} (1 + (|\mathcal{V}_i|-1)x),\]
which is exactly the $f$-polynomial of $\Delta(T_d(n-d))$.
\end{proof}

\begin{remark}\label{rem:PascalTriangle}
Lemma~\ref{prop:HVectorTuranFVectorTuran} gives a quick method to compute the values of $\binom{n}{0}_{\!d}, \binom{n}{1}_{\!d}, \dots, \binom{n}{d}_{\!d}$ via an iteration of (the reverse of) Stanley's trick; cf. \cite[p. 250]{book:Ziegler:LecturesOnPolytopes}. Suppose we know the values of $\binom{m}{0}_{\!d}, \binom{m}{1}_{\!d}, \dots, \binom{m}{d}_{\!d}$ for some $m>0$. Then we can compute the values of $\binom{m+d}{0}_{\!d}, \binom{m+d}{1}_{\!d}, \dots, \binom{m+d}{d}_{\!d}$ by using the following variant of Pascal's triangle: Construct a triangular array with rows labeled from $r=0$ to $r=d+1$, so that each $r$-th row has $r+1$ entries. Let the first entry of each row be $1$, and let the last entry of the $i$-th row be $\binom{m}{i}_{\!d}$ for every $i\in [d+1]$. Note that the last entry of the last row is always $0$. Compute the rest of the entries by using the rule that every entry is the sum of the two adjacent entries above it. The resulting first $(d+1)$ entries of the last row are precisely $\binom{m+d}{0}_{\!d}, \binom{m+d}{1}_{\!d}, \dots, \binom{m+d}{d}_{\!d}$ respectively. We can then determine $\binom{n}{0}_{\!d}, \binom{n}{1}_{\!d}, \dots, \binom{n}{d}_{\!d}$ by iterating this process, starting with $m$ being the unique integer in $[d]$ such that $m\equiv n\pmod{d}$.
\end{remark}

The following easy combinatorial identity is probably known; for completeness, we give a proof.
\begin{lemma}\label{prop:SumTuranCoefficients}
If $m,d$ are positive integers, then $\binom{m}{0}_{\!d} + \binom{m}{1}_{\!d} + \dots + \binom{m}{d}_{\!d} = \binom{m+d}{d}_{\!d}$.
\end{lemma}

\begin{proof}
Consider the variant Pascal's triangle $T$ as constructed in Remark \ref{rem:PascalTriangle}, and define the partial sum $s_i := \binom{m}{0}_{\!d} + \binom{m}{1}_{\!d} + \dots + \binom{m}{i}_{\!d}$ for each $0\leq i\leq d$. Since $\binom{m}{0}_{\!d} = 1$, it can be shown inductively that $s_i$ equals the second last entry of the $(i+1)$-th row of $T$; the assertion then follows from the observation that the second last entry of the $(d+1)$-th row equals $\binom{m+d}{d}_{\!d}$.
\end{proof}
We will need Zykov's generalization of the well-known Tur\'{a}n's theorem:
\begin{theorem}[{Zykov~\cite{Zykov1949}}]\label{thm:ZykovThm}
Let $\Delta$ be a flag $(d-1)$-complex with $n$ vertices. Then
\begin{equation*}
f_i(\Delta) \leq f_i(\Delta(T_d(n))) = \binom{n}{i+1}_{\!d}
\end{equation*}
for all $i\in [d-1]$. Moreover, this inequality is strict for all $i\in [d-1]$, unless $\Delta \cong \Delta(T_d(n))$.
\end{theorem}

We now collect useful results on canonical representations. Let $N, k, r$ be arbitrary positive integers satisfying $r\geq k$, and suppose that \eqref{eqn:CanonicalRepDefn} is the \textit{$(k,r)$-canonical representation} of $N$. By the uniqueness of $(k,r)$-canonical representations, the following integers are well-defined:
\begin{align}
\label{eqn:N+CanonRep} N_+^{(k,r)} &:=\binom{N_k+r}{k}_{\!r} + \binom{N_{k-1}+r-1}{k-1}_{\!r-1} + \dots + \binom{N_{k-s}+r-s}{k-s}_{\!r-s};\\
\label{eqn:N-CanonRep} N_-^{(k,r)} &:=\binom{N_k-r}{k}_{\!r} + \binom{N_{k-1}-(r-1)}{k-1}_{\!r-1} + \dots + \binom{N_{k-s}-(r-s)}{k-s}_{\!r-s}.
\end{align}
In particular, recall our convention that $\binom{n}{k}_{\!r} = 0$ for all integers $n$ such that $n<k$. The next lemma follows from the definition and uniqueness of canonical representations; its last item is the one that we will need later, and it follows from the previous items.
\begin{lemma}\label{lem:CanonRep+d}
Let $N$ be a positive integer whose $(k,r)$-canonical representation is given by \eqref{eqn:CanonicalRepDefn}.
Then:
\begin{enum*}
\item\label{item:N+CanonicalRep} The $(k,r)$-canonical representation of $N_+^{(k,r)}$ is given by \eqref{eqn:N+CanonRep}.
\item\label{item:N-CanonicalRep} If $N_k-r\geq k$, then the $(k,r)$-canonical representation of $N_-^{(k,r)}$ is
\[N_-^{(k,r)} =\binom{N_k-r}{k}_{\!r} + \binom{N_{k-1}-(r-1)}{k-1}_{\!r-1} + \dots + \binom{N_{k-s_0}-(r-s_0)}{k-s_0}_{\!r-s_0},\]
where $s_0 := \max\{t: t\leq s, N_{k-t} - (r-t) \geq k-t\}$. If instead $N_k-r < k$, then $N_-^{(k,r)} = 0$.
\item\label{item:RevLexInterpret}
Let $L$ be a positive integer whose $(k,r)$-canonical representation is given by
\[L = \binom{L_k}{k}_{\!r} + \binom{L_{k-1}}{k-1}_{\!r-1} + \dots + \binom{L_{k-t}}{k-t}_{\!r-t}.\]
Then $L<N$ if and only if either $t<s$ and $L_{k-i}=N_{k-i}$ for all $0\le i\le t$, or the minimal $i$ such that $L_{k-i}\neq N_{k-i}$ satisfies $L_{k-i}<N_{k-i}$.
\item\label{item:CanonicalRepTruncate} For any $0\leq m\leq t$,
\[\binom{L_{k-m}+1}{k-m}_{\!r-m} > \binom{L_{k-m}}{k-m}_{\!r-m}  + \binom{L_{k-m-1}}{k-m-1}_{\!r-m-1} + \dots + \binom{L_{k-t}}{k-t}_{\!r-t}.\]
\item\label{item:CanonicalRepCompare} $L_+^{(k,r)} \leq N$ if and only if $L \leq N_-^{(k,r)}$.
\end{enum*}
\end{lemma}
\begin{proof}
Parts \ref{item:N+CanonicalRep} and \ref{item:N-CanonicalRep} are immediate from the definition of $(k,r)$-canonical representations, while parts \ref{item:RevLexInterpret} and \ref{item:CanonicalRepTruncate} follow from the uniqueness of canonical representations. To prove part \ref{item:CanonicalRepCompare}, let $w := \min\{s,t\}$, and suppose that $L_+^{(k,r)} \leq N$. This implies
\begin{equation}\label{eqn:L+truncateNtruncate}
\binom{L_k+r}{k}_{\!r} + \dots + \binom{L_{k-w}+r-w}{k-w}_{\!r-w} \leq \binom{N_k}{k}_{\!r} + \dots + \binom{N_{k-w}}{k-w}_{\!r-w}.
\end{equation}
Indeed, \eqref{eqn:L+truncateNtruncate} clearly holds if $s\leq t$, while the case $s>t$ is implied by part \ref{item:RevLexInterpret}, since both expressions in \eqref{eqn:L+truncateNtruncate} are $(k,r)$-canonical representations.

If equality holds in \eqref{eqn:L+truncateNtruncate}, then the assumption $L_+^{(k,r)} \leq N$ forces $s\geq t$. Consequently, the uniqueness of $(k,r)$-canonical representations yields $L_{k-i}+(r-i) = N_{k-i}$ for all $0\leq i\leq t=w$, hence
\[L = \binom{L_k}{k}_{\!r} + \dots + \binom{L_{k-t}}{k-t}_{\!r-t} = \binom{N_k-r}{k}_{\!r} + \dots + \binom{N_{k-t}-(r-t)}{k-t}_{\!r-t} \leq N_-^{(k,r)}.\]

If instead the inequality in \eqref{eqn:L+truncateNtruncate} is strict, then by part \ref{item:RevLexInterpret}, $L_{k-j}+r-j<N_{k-j}$ for the smallest $0\leq j\leq w$ such that $L_{k-j}+r-j\neq N_{k-j}$; this gives
\begin{equation}\label{LtruncateN-truncate}
\binom{L_k}{k}_{\!r} + \dots + \binom{L_{k-w}}{k-w}_{\!r-w} < \binom{N_k-r}{k}_{\!r} + \dots + \binom{N_{k-w}-(r-w)}{k-w}_{\!r-w}.
\end{equation}
The case $s\geq t$ clearly gives $L <N_-^{(k,r)}$, while for the case $s<t$, part \ref{item:CanonicalRepTruncate} yields
\begin{equation}\label{eqn:LLroundedUp}
L < \binom{L_k}{k}_{\!r} + \dots + \binom{L_{k-(s-1)}}{k-(s-1)}_{\!r-(s-1)} + \binom{L_{k-s}+1}{k-s}_{\!r-s},
\end{equation}
where by definition, the $(k,r)$-canonical representation of the right-hand expression in \eqref{eqn:LLroundedUp} is
\[\binom{L_k}{k}_{\!r} + \dots + \binom{L_{k-(s'-1)}}{k-(s'-1)}_{\!r-(s'-1)} + \binom{L_{k-s'}+1}{k-s'}_{\!r-s'}\]
for some uniquely determined $0\leq s'\leq s$. Thus by part \ref{item:RevLexInterpret}, the right-hand side of \eqref{eqn:LLroundedUp} is at most the right-hand side of \eqref{LtruncateN-truncate}, which gives $L< N_-^{(k,r)}$ in this case. Finally, an analogous argument shows that $L\leq N_-^{(k,r)}$ implies $L_+^{(k,r)} \leq N$.
\end{proof}

Since the expression \eqref{eqn:CanonicalRepDefn} is unique (when $N,k,r$ are fixed), we can define the functions $\partial_k^{(r)}$ and $\partial_{(r)}^{k}$ on positive integers as follows.
\begin{align*}
\partial_{k}^{(r)}(N) &:= \binom{N_k}{k-1}_{\!r} + \binom{N_{k-1}}{k-2}_{\!r-1} + \dots + \binom{N_{k-s}}{k-s-1}_{\!r-s}.\\
\partial_{(r)}^{k}(N) &:= \binom{N_k}{k+1}_{\!r} + \binom{N_{k-1}}{k}_{\!r-1} + \dots + \binom{N_{k-s}}{k-s+1}_{\!r-s}.
\end{align*}
For convenience, define $\partial_{k}^{(r)}(0) = \partial_{(r)}^{k}(0) = 0$.

The following two lemmas follow easily from the definition of canonical representations.

\begin{lemma}\label{lem:iteratedDifferentialOperator}
If the $(k,r)$-canonical representation of $N$ is given by \eqref{eqn:CanonicalRepDefn}, then for all $0\leq j < k$,
\[\partial_{k-j}^{(r)}(\cdots\partial_{k-1}^{(r)}(\partial_{k}^{(r)}(N))) = \binom{N_k}{k-j-1}_{\!r} + \binom{N_{k-1}}{k-j-2}_{\!r-1} + \dots + \binom{N_{k-s}}{k-j-s-1}_{\!r-s}.\]
\end{lemma}

\begin{lemma}\label{lem:iteratedDifferentialOperatorSecond}
If the $(k,r)$-canonical representation of $N$ is given by \eqref{eqn:CanonicalRepDefn}, then for all $j\geq 0$,
\[\partial^{k+j}_{(r)}(\cdots\partial^{k+1}_{(r)}(\partial^{k}_{(r)}(N))) = \binom{N_k}{k+j+1}_{\!r} + \binom{N_{k-1}}{k+j}_{\!r-1} + \dots + \binom{N_{k-s}}{k+j-s+1}_{\!r-s}.\]
\end{lemma}

\begin{theorem}[Frankl--F\"{u}redi--Kalai~\cite{FFK1988}]\label{thm:FFK}
Let $\mathbf{f} = (f_{-1}, f_0, \dots, f_{d-1})$ be a $(d+1)$-tuple of positive integers for some $d\leq r$. Then the following are equivalent.
\begin{enum*}
\item $\mathbf{f}$ is the $f$-vector of an $r$-colorable complex.
\item $\mathbf{f}$ is the $f$-vector of a revlex $r$-colorable complex with respect to $\Pi_r$.
\item $f_{-1} = 1$, and $f_{k-1} \geq \partial_{k+1}^{(r)}(f_k)$ for all $k\in [d-1]$.
\item $f_{-1} = 1$, and $\partial_{(r)}^k(f_{k-1}) \geq f_k$ for all $k\in [d-1]$.
\end{enum*}
\end{theorem}

\begin{theorem}[{Frohmader~\cite{Frohmader2008:FaceVectorsOfFlagComplexes}}]\label{thm:Frohmader}
Let $\Delta$ be a flag complex of dimension $<d$. Then there exists a revlex $d$-colorable complex $\Gamma$ with the same $f$-vector as $\Delta$.
\end{theorem}

\section{Homology of color-shifted balanced complexes}\label{sec:HomologyColorShiftedBalanced}
Throughout this section, assume that the ground set $U$ has size $\geq d$, and fix an ordered partition $\pi = (U_1, \dots, U_d)$ of $U$, such that the elements of each $U_i$ are labeled as $u_{i,j}$, i.e. $U_i = \{u_{i,j}\}_{j\in [\lambda_j]}$ for some $\lambda_j \in \mathbb{P}$ (if $U_i$ is finite), or $U_i = \{u_{i,j}\}_{j\in \mathbb{P}}$ (if $U_i$ is countably infinite). Assume that $u_{i,1} < u_{i,2} < \dots$, i.e. each $U_i$ is linearly ordered.

\begin{theorem}[{Babson--Novik~\cite[Thm. 5.7]{BabsonNovik2006:FaceNumerbsNongenericInitialIdeals}}]\label{thm:PureColorShiftedBalancedBetti}
If $(\Delta, \pi)$ is a pure color-shifted balanced $(d-1)$-complex, then $\beta_i(\Delta) = 0$ for all $i<d-1$, and
\begin{equation}\label{eqn:TopBettiNumber}
\beta_{d-1}(\Delta) = |\{F\in \mathcal{F}_{d-1}(\Delta): u_{i,1} \not\in F \text{ for all }i\in [d]\}|.
\end{equation}
\end{theorem}

The original statement of \cite[Thm. 5.7]{BabsonNovik2006:FaceNumerbsNongenericInitialIdeals} asserts that if $(\Delta, \pi)$ is a (not necessarily pure) color-shifted balanced complex, then $\beta_j(\Delta)$ equals the number of $j$-dimensional facets $F$ of $\Delta$ satisfying $u_{i,1} \not\in F$ for all $i\in [d]$. However, as pointed out by Murai~\cite{Murai2008:BettiNumbersSquarefreeStronglyColorStable}, this original assertion is incorrect if $\Delta$ is not pure: Murai gave a non-pure counter-example to the original assertion, explained why the proof of \cite[Thm. 5.7]{BabsonNovik2006:FaceNumerbsNongenericInitialIdeals} requires the assumption that $\Delta$ is pure, and gave a different proof in the pure case; see \cite[Prop. 4.2]{Murai2008:BettiNumbersSquarefreeStronglyColorStable}. Note that Theorem \ref{thm:PureColorShiftedBalancedBetti}, as we have stated here, includes this correction.

Although Theorem \ref{thm:PureColorShiftedBalancedBetti} is not true when $\Delta$ is not pure, the formula for the top-dimensional Betti number still holds in the non-pure case (see Corollary \ref{cor:NonpureColorShiftedBalancedTopBetti} below); its proof is based on the simple observation that if $\Gamma$ is the pure subcomplex of $\Delta$ generated by $\mathcal{F}_d(\Delta)$, then both $\Delta$ and $\Gamma$ have the same $d$-chains, and hence the same $d$-cycles, which implies that $\widetilde{H}_d(\Delta) \cong \widetilde{H}_d(\Gamma)$. The following two corollaries are immediate consequences of this observation.

\begin{corollary}\label{cor:SameTopDimFacesSameTopBetti}
If $\Delta, \Delta'$ are simplicial $d$-complexes such that $\mathcal{F}_d(\Delta) = \mathcal{F}_d(\Delta')$, then $\beta_d(\Delta) = \beta_d(\Delta')$.
\end{corollary}

\begin{corollary}\label{cor:NonpureColorShiftedBalancedTopBetti}
If $(\Delta, \pi)$ is a (not necessarily pure) color-shifted balanced $(d-1)$-complex, then $\beta_{d-1}(\Delta)$ satisfies \eqref{eqn:TopBettiNumber}.
\end{corollary}

Given a simplicial complex $\Delta$ and any finite subset $\mathcal{A} = \{a_1, \dots, a_r\} \subseteq \mathbb{N}$, we say that $\Delta$ is \textit{$\mathcal{A}$-facet-free} if every $F\in \Facet(\Delta)$ satisfies $\dim F\not\in \mathcal{A}$. Assume that $a_1 > \dots > a_r$. Let $\Gamma_0 = \Delta$, and iteratively define $\Gamma_i = \Gamma_{i-1}\backslash \big(\mathcal{F}_{a_i}(\Gamma_{i-1}) \cap \Facet(\Gamma_{i-1})\big)$ for $i = 1, \dots, r$ (in this order). The final simplicial complex $\Gamma_r$ we get shall be called the \textit{$\mathcal{A}$-facet-free reduction} of $\Delta$. Note that by construction, $\Gamma_r$ is $\mathcal{A}$-facet-free, and $\mathcal{F}_k(\Gamma_r) = \mathcal{F}_k(\Delta)$ for all $k\in \mathbb{N}\backslash\mathcal{A}$. The following corollary is an immediate consequence of Corollary \ref{cor:SameTopDimFacesSameTopBetti}.

\begin{corollary}\label{cor:FacetFreeReductionTopBetti}
Let $\Delta$ be a simplicial complex of dimension $\leq d$. If $\Gamma$ is the $\mathcal{A}$-facet-free reduction of $\Delta$ for some finite $\mathcal{A} \subseteq \mathbb{N}$ satisfying $d\not\in \mathcal{A}$, then $\beta_{d}(\Gamma) = \beta_{d}(\Delta)$.
\end{corollary}

Recall that $\Pi_d = (\Pi_{d,1}, \dots, \Pi_{d,d})$, where $\Pi_{d,i} := \{k\in \mathbb{P}: k \equiv i\pmod{d}\}$ for each $i\in [d]$.

\begin{theorem}\label{thm:RevlexComplexMaximizesBettiNum}
Let $1\leq k\leq d$ be integers, and let $\mathcal{B}_N^{k,d}$ be the set of all balanced $(d-1)$-complexes $\Delta$ satisfying $f_{k-1}(\Delta) = N>0$. If $N$ has the $(k,d)$-canonical representation
\[N = \binom{N_{d}}{k}_{\!d} + \binom{N_{d-1}}{k-1}_{\!d-1} + \dots + \binom{N_{d-s}}{k-s}_{\!d-s},\]
then for every $\Delta \in \mathcal{B}_N^{k,d}$,
\begin{equation}\label{eqn:IneqTopBettiNum}
\beta_{d-1}(\Delta) \leq \binom{N_{d}-d}{d}_{\!d} + \binom{N_{d-1}-(d-1)}{d-1}_{\!d-1} + \dots + \binom{N_{d-s} - (d-s)}{d-s}_{\!d-s}.
\end{equation}
Furthermore, if $\Delta \in \mathcal{B}_N^{k,d}$ is a revlex balanced complex with respect to $\Pi_d$ that satisfies
\[f_{d-1}(\Delta) = \binom{N_{d}}{d}_{\!d} + \binom{N_{d-1}}{d-1}_{\!d-1} + \dots + \binom{N_{d-s}}{d-s}_{\!d-s},\]
then equality holds in \eqref{eqn:IneqTopBettiNum}.
\end{theorem}

\begin{proof}
First, for any balanced complex $\Gamma$, it follows from \cite[Thm. 0.1]{Murai2008:BettiNumbersSquarefreeStronglyColorStable} (cf. \cite[Cor. 3.4]{Murai2008:BettiNumbersSquarefreeStronglyColorStable} and Hochster's formula~\cite[Thm. 5.5.1]{book:BrunsHerzog:CMRings}) that $\beta_i(\Gamma) \leq \beta_i(\widetilde{\Delta}_{\prec}(\Gamma))$ for all $i\in \mathbb{N}$, where $\widetilde{\Delta}_{\prec}(\Gamma)$ denotes the colored algebraic shifting of $\Gamma$; see \cite{BabsonNovik2006:FaceNumerbsNongenericInitialIdeals} for a precise definition of colored algebraic shifting. In particular, \cite[Thm. 5.6]{BabsonNovik2006:FaceNumerbsNongenericInitialIdeals} tells us that $\widetilde{\Delta}_{\prec}(\Gamma)$ is color-shifted and has the same $f$-vector as $\Gamma$.

By Frankl--F\"{u}redi--Kalai's theorem (Theorem \ref{thm:FFK}) and Lemma \ref{lem:iteratedDifferentialOperatorSecond}, we have
\[f_{d-1}(\Delta) \leq \binom{N_{d}}{d}_{\!d} + \binom{N_{d-1}}{d-1}_{\!d-1} + \dots + \binom{N_{d-s}}{d-s}_{\!d-s}.\]
This means that both assertions for arbitrary $k\in [d]$ follows from the special case $k=d$. Henceforth, we shall assume that $k=d$.

Consider an arbitrary $\Delta \in \mathcal{B}_N^{d,d}$, and assume that $\Delta$ is balanced with respect to $\pi = (U_1, \dots, U_d)$. By taking its colored algebraic shifting if necessary, we can assume without loss of generality that $(\Delta, \pi)$ is color-shifted.

Define $\widehat{U} := \{u_{i,j} \in U: j\neq 1\} \subseteq U$, and define $\widehat{\Delta} := \{F\cap \widehat{U}: F\in \mathcal{F}_{d-1}(\Delta)\}$. We claim that $\widehat{\Delta}$ is a subcomplex of $\Delta$. First, note that there is a bijection $\phi: \widehat{\Delta} \to \mathcal{F}_{d-1}(\Delta)$ given by
\[\widehat{F} \mapsto \widehat{F} \cup \bigcup_{\substack{i\in [d]: \widehat{F} \cap U_i = \emptyset}} u_{i,1}.\]
Consider any pair $(F', F) \in \Delta \times \Delta$ satisfying $F'\subseteq F$ and $F\in \widehat{\Delta}$. Suppose $\dim F' = q-1$, and let $u_{i_1,j_1}, \dots, u_{i_{d-q}, j_{d-q}}$ be the $d-q$ uniquely determined vertices in $\mathcal{V}(\phi(F))\backslash \mathcal{V}(F')$.
By assumption, $\Delta$ is $d$-colored with respect to $\pi$, so $i_1, \dots, i_{d-q}$ must be distinct. Since $\Delta$ is color-shifted, we then infer that $F'\cup \bigcup_{t\in [d-q]} u_{i_t,1} \in \mathcal{F}_{d-1}(\Delta)$, hence $F'\in \widehat{\Delta}$ by definition. Since this holds for all possible pairs $(F',F)$, it follows that $\widehat{\Delta}$ is a subcomplex of $\Delta$ as claimed.

We now split into two cases (i): $f_{d-1}(\widehat{\Delta}) \geq 1$; and (ii): $f_{d-1}(\widehat{\Delta}) = 0$. First, consider case (i), and suppose that
\begin{equation}\label{eqn:CanonicalRepNPrime}
f_{d-1}(\widehat{\Delta}) = \binom{N_d'}{d}_{\!d} + \binom{N_{d-1}'}{d-1}_{\!d-1} + \dots + \binom{N_{m}'}{m}_{\!m}
\end{equation}
is the $(d,d)$-canonical representation of $f_{d-1}(\widehat{\Delta})$. Since $\widehat{\Delta}$ is $d$-colorable, it follows from Lemma \ref{lem:iteratedDifferentialOperator} and Theorem \ref{thm:FFK} that
\begin{equation}\label{eqn:lowerFacesIneq}
f_{d-1-j}(\widehat{\Delta}) \geq \binom{N_d'}{d-j}_{\!d} + \binom{N_{d-1}'}{d-1-j}_{\!d-1} + \dots + \binom{N_{m}'}{m-j}_{\!m}
\end{equation}
for all $j\in [d-1]$.

Note that the bijection $\phi$ implies $N = \sum_{i=-1}^{d-1} f_i(\widehat{\Delta})$, so it follows from \eqref{eqn:lowerFacesIneq} that
\begin{equation*}
\sum_{j=m}^{d} \bigg[\binom{N_j'}{0}_{\!j} + \binom{N_j'}{1}_{\!j} + \dots + \binom{N_j'}{j}_{\!j}\bigg] \leq N.
\end{equation*}
Consequently, by Proposition \ref{prop:SumTuranCoefficients}, we infer that
\begin{equation}\label{eqn:NprimeVersusN}
L:=\sum_{j=m}^d \binom{N_j'+j}{j}_{\!j} \leq N = \sum_{j=d-s}^d \binom{N_j}{j}_{\!j}.
\end{equation}
Apply Lemma \ref{lem:CanonRep+d} to the two $(d,d)$-canonical representations in \eqref{eqn:NprimeVersusN} to conclude that $L_-^{(d,d)}\leq N_-^{(d,d)}$. Now, by Corollary \ref{cor:NonpureColorShiftedBalancedTopBetti}, $\beta_{d-1}(\Delta) = f_{d-1}(\widehat{\Delta})=L_-^{(d,d)}$, hence
\eqref{eqn:IneqTopBettiNum} follows.

Next, let $\Delta$ be revlex with respect to $\pi = \Pi_d$, i.e. we identify each subset $U_i$ with $\Pi_{d,i} = \{k\in \mathbb{P}: k \equiv i\pmod{d}\}$. Then by definition, $\widehat{\Delta}$ is revlex with respect to $(\Pi_{d,1}\backslash\{1\}, \dots, \Pi_{d,d}\backslash\{d\})$, therefore a simple counting argument yields $f_{d-1}(\widehat{\Delta}) = \sum_{j=d-s}^d \binom{N_j-j}{j}_{\!j}$ as desired.

Finally, we turn to case (ii). Note that \eqref{eqn:IneqTopBettiNum} is trivially true. Note further that $f_{d-1}(\widehat{\Delta}) = 0$ forces $N < \binom{2d}{d}_{\!d}$. Indeed, if instead there exists some $G = u_{1,j_1} \cup \dots u_{d,j_d} \in \mathcal{F}_{d-1}(\widehat{\Delta})$, then the definition of $\widehat{\Delta}$ implies that $j_t \geq 2$ for all $t\in [d]$. Since $\Delta$ is color-shifted by assumption, we would then get that
\[\Big\{u_{1,j_1'} \cup \dots \cup u_{d,j_d'}: j_t' \in \{1,2\}\text{ for each }t\in [d]\Big\}\]
is a set of $2^d = \binom{2d}{d}_{\!d}$ distinct $(d-1)$-faces contained in $\Delta$. Consequently, we infer our assertion for balanced revlex complexes in case (ii) from Corollary \ref{cor:NonpureColorShiftedBalancedTopBetti}, by noting that all revlex $(d-1)$-balanced complexes $\Delta$ satisfying $f_{d-1}(\Delta) < \binom{2d}{d}_{\!d}$ would by definition not have any $(d-1)$-faces in $\widehat{\Delta}$, i.e. $\beta_{d-1}(\Delta) = f_{d-1}(\widehat{\Delta}) = 0$ in this revlex case.
\end{proof}

\section{Homology of flag complexes}\label{sec:HomologyFlagComplexes}
\begin{lemma}\label{lem:standardMayerViet}
Let $v$ be a vertex in a simplicial complex $\Delta$. Then $\beta_k(\Delta) \leq \beta_k(\Ast_{\Delta}(v)) + \beta_{k-1}(\Lk_{\Delta}(v))$ for all $k\geq 0$.
\end{lemma}

\begin{proof}
Note that $\overline{\St}_{\Delta}(v)$ is contractible, so it has trivial reduced homology. Since $\Ast_{\Delta}(v) \cap \overline{\St}_{\Delta}(v) = \Lk_{\Delta}(v)$, the Mayer--Vietoris sequence for the decomposition $\Delta = \Ast_{\Delta}(v) \cup \overline{\St}_{\Delta}(v)$ thus yields the exact sequence
\[\cdots \longrightarrow \widetilde{H}_k(\Ast_{\Delta}(v)) \overset{\partial}{\longrightarrow} \widetilde{H}_k(\Delta) \overset{\delta}{\longrightarrow} \widetilde{H}_{k-1}(\Lk_{\Delta}(v)) \longrightarrow \cdots,\]
so by exactness at $\widetilde{H}_k(\Delta)$,
\[\beta_k(\Delta) = \dim_{\Bbbk}\Ker \delta + \dim_{\Bbbk}\Img \delta = \dim_{\Bbbk}\Img \partial + \dim_{\Bbbk}\Img \delta \leq \beta_k(\Ast_{\Delta}(v)) + \beta_{k-1}(\Lk_{\Delta}(v))\]
for all $k\geq 0$.
\end{proof}

\begin{theorem}\label{TopBettiIneqRevlex}
Let $\Delta$ be a flag $d$-complex. Then the minimal (w.r.t. inclusion) revlex $(d+1)$-colorable complex $\Gamma$ that satisfies $f_d(\Gamma) = f_d(\Delta)$ must  satisfy $\beta_d(\Gamma) \geq \beta_d(\Delta)$.
\end{theorem}

\begin{proof}
We shall prove by induction on $d$. The base case $d=0$ is trivially true, so assume that $d\geq 1$. Let $v_0 \in \mathcal{V}(\Delta)$ be any vertex contained in the most number of $d$-faces in $\Delta$, and suppose that $\mathcal{V}(\Delta)\backslash \mathcal{V}(\overline{\St}_{\Delta}(v_0))$ has cardinality $s$. Choose an arbitrary linear order $v_1, \dots, v_s$ of the vertices in $\mathcal{V}(\Delta)\backslash \mathcal{V}(\overline{\St}_{\Delta}(v_0))$ (if any). Next, let $\Delta_{s+1} := \Delta$, and if $s\geq 1$, then iteratively define $\Delta_i := \Ast_{\Delta_{i+1}}(v_i)$ and $a_i := f_{d-1}(\Lk_{\Delta_{i+1}}(v_i))$ for each $i\in [s]$.

By the definition of $s$, we infer that $\mathcal{V}(\Delta_1) = \mathcal{V}(\overline{\St}_{\Delta}(v_0))$. Since $\Delta$ is a flag complex, we know that $\overline{\St}_{\Delta}(v_0)$ is an induced subcomplex of $\Delta$ (cf. \cite[Rem. 3.5]{ConstantinescuVarbaro2013:hvectorsCMFlagComplexes}), thus $\Delta_1 = \overline{\St}_{\Delta}(v_0)$. Notice that if $s=0$, then $\Delta = \overline{\St}_{\Delta}(v_0)$, which implies $\beta_d(\Delta) = 0$, so our assertion becomes trivially true. Henceforth, we shall assume that $s\geq 1$. Let $\Delta_0 := \Ast_{\Delta_1}(v_0) = \Lk_{\Delta}(v_0)$, let $a_0 := f_{d-1}(\Lk_{\Delta_1}(v_0)) = f_{d-1}(\Lk_{\Delta}(v_0))$, and note that
\[\Lk_{\Delta}(v_0) = \Delta_0 \subsetneq \Delta_1 \subsetneq \dots \subsetneq \Delta_{s+1} = \Delta.\]

For each $0\leq i\leq s$, note that $f_d(\Delta_{i+1}) - f_d(\Delta_i) = a_i$ by construction. Since $\dim \Delta = d$ implies $f_d(\Delta_0) = f_d(\Lk_{\Delta}(v_0)) = 0$, it then follows that $f_d(\Delta) = a_0 + a_1 + \dots + a_s$. Also, since $\Delta_{i+1} \subseteq \Delta$, it follows from our choice of $v_0$ that
\begin{equation}\label{eqn:UpperBoundAZero}
a_0 = f_{d-1}(\Lk_{\Delta}(v_0)) \geq f_{d-1}(\Lk_{\Delta}(v_i)) \geq f_{d-1}(\Lk_{\Delta_{i+1}}(v_i)) = a_i.
\end{equation}

Since each $\Lk_{\Delta_{i+1}}(v_i)$ is a flag complex of dimension $< d$, we infer that there exists a revlex $d$-colorable complex $\Sigma_i$ satisfying $f_{d-1}(\Sigma_i) = f_{d-1}(\Lk_{\Delta_{i+1}}(v_i)) = a_i$ and $\beta_{d-1}(\Sigma_i) \geq \beta_{d-1}(\Lk_{\Delta_{i+1}}(v_i))$; the case $a_i \geq 1$ (i.e. $\dim \Lk_{\Delta_{i+1}}(v_i) = d-1$) follows from the induction hypothesis, while the case $a_i = 0$ (hence $\beta_{d-1}(\Lk_{\Delta_{i+1}}(v_i)) = 0$) is trivial.

By Lemma \ref{lem:standardMayerViet}, $\beta_d(\Delta_{i+1}) \leq \beta_d(\Delta_i) + \beta_{d-1}(\Lk_{\Delta_{i+1}}(v_i))$ for every $i\in [s]$. Since $\Delta_1 = \overline{\St}_{\Delta}(v_0)$ is contractible, we get $\beta_d(\Delta_1) = 0$. Thus,
\begin{equation}\label{eqn:BettiDelta}
\beta_d(\Delta) = \beta_d(\Delta_{s+1}) \leq \sum_{i\in [s]} \beta_{d-1}(\Lk_{\Delta_{i+1}}(v_i)) \leq \sum_{i\in [s]} \beta_{d-1}(\Sigma_i).
\end{equation}

Without loss of generality, all $\Sigma_i$'s are defined on a common ordered partition $(U_1, \dots, U_d)$ of some common ground set $U$, where each subset $U_i$ has an arbitrarily large finite size, whose elements are linearly ordered by $u_{i,1} < u_{i,2} < \dots$. Next, let $U_{d+1}$ be another set with an arbitrarily large finite size, whose elements are linearly ordered by $u_0 < u_1 < u_2 < \dots$, and extend the ground set to the disjoint union $U' = U \sqcup U_{d+1}$. Now, define the simplicial complex
\begin{equation}\label{eqn:Sigma}
\Sigma := (\Sigma_0 * \langle u_0\rangle) \cup (\Sigma_1 * \langle u_1\rangle) \cup \dots \cup (\Sigma_s * \langle u_s\rangle).
\end{equation}
Note that $\Sigma$ is $(d+1)$-colorable with respect to the ordered partition $(U_1, \dots, U_d, U_{d+1})$ of $U'$. Note also that by construction,
\begin{equation}\label{eqn:TopFSigmaDelta}
f_d(\Sigma) = a_0 + a_1 + \dots + a_s = f_d(\Delta).
\end{equation}

We shall now compute $\beta_d(\Sigma)$. For every $0\leq j\leq s$, define $\Sigma_j' := (\Sigma_0 * \langle u_0\rangle) \cup \dots \cup (\Sigma_j * \langle u_j\rangle)$. Next, for each $i\in [s]$, let $\Sigma_i'' := \Sigma_i * \langle u_0, u_i\rangle$. Recall that \eqref{eqn:UpperBoundAZero} says $a_i = f_{d-1}(\Sigma_i) \leq f_{d-1}(\Sigma_0) = a_0$ for all $i\in [s]$. Since $\Sigma_0, \dots, \Sigma_s$ are revlex $d$-colorable complexes, it follows that
$\mathcal{F}_{d-1}(\Sigma_i) \subseteq \mathcal{F}_{d-1}(\Sigma_0)$, so we infer that $\Sigma_{i-1}' \cup \Sigma_i'' = \Sigma_i'$ and $\Sigma_{i-1}' \cap \Sigma_i'' = \Sigma_i * \langle u_0\rangle$ for all $i\in [s]$. Thus, the Mayer--Vietoris sequence for the decomposition $\Sigma_i' = \Sigma_{i-1}' \cup \Sigma_i''$ yields the exact sequence
\begin{equation}\label{eqn:MVsuspension}
\cdots \longrightarrow \widetilde{H}_d(\Sigma_i * \langle u_0\rangle) \longrightarrow \widetilde{H}_d(\Sigma_{i-1}') \oplus \widetilde{H}_d(\Sigma_i'') \longrightarrow \widetilde{H}_d(\Sigma_i') \longrightarrow \widetilde{H}_{d-1}(\Sigma_i * \langle u_0\rangle) \longrightarrow \cdots
\end{equation}
The cone $\Sigma_i * \langle u_0\rangle$ is contractible and so has trivial reduced homology. Also, $\Sigma_i''$ is a suspension over $\Sigma_i$, which yields $\widetilde{H}_d(\Sigma_i'') \cong \widetilde{H}_{d-1}(\Sigma_i)$. Consequently, \eqref{eqn:MVsuspension} implies that $\beta_d(\Sigma_i') = \beta_d(\Sigma_{i-1}') + \beta_{d-1}(\Sigma_i)$ for all $i\in [s]$. Since $\beta_d(\Sigma_0') = \beta_d(\Sigma_0 * \langle u_0\rangle) = 0$, it then follows that
\begin{equation}\label{eqn:BettiSigma}
\beta_d(\Sigma) = \beta_d(\Sigma_s') = \sum_{i\in [s]} \beta_{d-1}(\Sigma_i).
\end{equation}

Now, from \eqref{eqn:TopFSigmaDelta}, \eqref{eqn:BettiDelta} and \eqref{eqn:BettiSigma}, we conclude that $\Sigma$ is a $(d+1)$-colorable complex satisfying $f_d(\Sigma) = f_d(\Delta)$ and $\beta_d(\Sigma) \geq \beta_d(\Delta)$. Therefore, by Theorem \ref{thm:RevlexComplexMaximizesBettiNum}, the minimal revlex $(d+1)$-colorable complex $\Gamma$ satisfying $f_d(\Gamma) = f_d(\Delta)$ also satisfies $\beta_d(\Gamma) \geq \beta_d(\Delta)$, which completes the induction step.
\end{proof}

\begin{remark}\label{rem:VertexDecompProof}
Theorem \ref{TopBettiIneqRevlex} can alternatively be proven by showing that $\Sigma$, as constructed in \eqref{eqn:Sigma}, is vertex-decomposable. Recall that a (not necessarily pure) simplicial complex $\Delta'$ is called \textit{vertex-decomposable} if either $\Delta'$ is a simplex (possibly the trivial complex $\{\emptyset\}$); or there exists a vertex $v \in \mathcal{V}(\Delta)$ such that both $\Ast_{\Delta'}(v)$ and $\Lk_{\Delta'}(v)$ are vertex-decomposable, and no facet of $\Lk_{\Delta'}(v)$ is a facet of $\Ast_{\Delta'}(v)$. Such a vertex $v$ (if it exists) is called a \textit{shedding vertex}. First of all, by using Corollary \ref{cor:FacetFreeReductionTopBetti}, we may assume that each $\Sigma_i$ is $([d-2]\cup\{0\})$-facet-free. In particular, $\Sigma_i$ is a pure $(d-1)$-complex if $a_i\geq 1$, and $\Sigma_i = \{\emptyset\}$ is the trivial complex if $a_i = 0$. Let $\Gamma_{s+1} := \Sigma$, and iteratively define $\Gamma_i := \Ast_{\Gamma_{i+1}}(u_i)$ for each $i\in [s]$. Each $\Lk_{\Gamma_{i+1}}(u_i) = \Sigma_i$ is vertex-decomposable, and note also that $(\Sigma_0 * \langle u_0\rangle) \subseteq \Gamma_i$, so none of the faces in $\mathcal{F}_{d-1}(\Sigma_i)$ are facets of $\Gamma_i$, which establishes the vertex-decomposability of $\Sigma$. Next, observe that for any vertex-decomposable simplicial complex $\Delta'$ with shedding vertex $v$, the identity $\beta_k(\Delta') = \beta_k(\Ast_{\Delta'}(v)) + \beta_{k-1}(\Lk_{\Delta'}(v))$ holds for all $k\in \mathbb{N}$; this follows from \cite[Thm. 11.3]{BjornerWachs1997:NonpureShellableII} and \cite[Thm. 4.1]{BjornerWachs1996:NonpureShellableI}. Thus, we can infer \eqref{eqn:BettiSigma} by iteratively applying this identity on $\Sigma$ using the sequence $u_s, u_{s-1}, \dots, u_1$ of shedding vertices.
\end{remark}

\begin{remark}
A third proof of Theorem \ref{TopBettiIneqRevlex} involves showing that $\Sigma$ is color-shifted. Similarly, as in Remark \ref{rem:VertexDecompProof}, we may assume that each $\Sigma_i$ is $([d-2]\cup\{0\})$-facet-free. Let $\sigma: [s]\to [s]$ be any permutation such that $a_{\sigma(s)} \geq a_{\sigma(s-1)} \geq \dots \geq a_{\sigma(1)}$, and replace $u_i$ in $U_{d+1}$ by $u_{\sigma(i)}$ for all $i\in [s]$. Then $\Sigma$ is color-shifted with respect to the resulting new ordered partition $(U_1, \dots, U_d, U_{d+1})$, thus we can infer \eqref{eqn:BettiSigma} from \eqref{eqn:TopBettiNumber}.
\end{remark}

\begin{proof}[\normalfont\textbf{Proof of Theorem \ref{thm:FrohmaderWithHomology}.}]
By Theorem \ref{TopBettiIneqRevlex}, the minimal revlex balanced $d$-complex $\Sigma$ that satisfies $f_d(\Sigma) = f_d(\Delta)$ must satisfy $\beta_d(\Sigma) \geq \beta_d(\Delta)$. By Frohmader's theorem (Theorem \ref{thm:Frohmader}), there exists a (unique) revlex balanced complex $\Gamma$ with the same $f$-vector as $\Delta$.
Then $\Sigma\subseteq \Gamma$ and $\mathcal{F}_d(\Gamma) = \mathcal{F}_d(\Sigma)$, thus by Corollary \ref{cor:SameTopDimFacesSameTopBetti}, we conclude that $\beta_d(\Gamma) = \beta_d(\Sigma) \geq \beta_d(\Delta)$.
\end{proof}

\begin{proof}[\normalfont\textbf{Proof of Theorem \ref{thm:FlagTopBettiUpperBound}.}]
By applying Theorem \ref{thm:FrohmaderWithHomology}, we infer there is a revlex balanced complex $\Sigma$ that satisfies $f(\Sigma) = f(\Delta)$ and $\beta_{d-1}(\Sigma) \geq \beta_{d-1}(\Delta)$. So by Frankl--F\"{u}redi--Kalai's theorem (Theorem \ref{thm:FFK}) and Lemma \ref{lem:iteratedDifferentialOperatorSecond},
\[f_{d-1}(\Delta) \leq \binom{N_{d}}{d}_{\!d} + \binom{N_{d-1}}{d-1}_{\!d-1} + \dots + \binom{N_{d-s}}{d-s}_{\!d-s}.\]
The assertion then follows from Theorem \ref{thm:RevlexComplexMaximizesBettiNum}.
\end{proof}

\section{Lower bounds on face numbers}\label{sec:LowerBoundsFaceNum}

\begin{proof}[\normalfont\textbf{Proof of Theorem \ref{thm:FlagLowerBoundfVecGivenTopBetti}.}]
First of all, to prove the first assertion, it suffices to show \eqref{eqn:fVecFlagLowerBound} for $i=d$, since by Frohmader's theorem (Theorem \ref{thm:Frohmader}), Frankl--F\"{u}redi--Kalai's theorem (Theorem \ref{thm:FFK}) and Lemma \ref{lem:iteratedDifferentialOperator}, we would then get \eqref{eqn:fVecFlagLowerBound} for all $i\geq 0$. Let
\[f_{d-1}(\Delta) = \binom{a_d'}{d}_{\!d} + \binom{a_{d-1}'}{d-1}_{\!d-1} + \dots + \binom{a_m'}{m}_{\!m}\]
be the $(d,d)$-canonical representation of $f_{d-1}(\Delta)$. By Theorem \ref{thm:FlagTopBettiUpperBound}, we get
\begin{equation}\label{eqn:TopBettiFlagUpperBoundNearCanonicalRep}
\beta_{d-1}(\Delta) \leq \binom{a_d'-d}{d}_{\!d} + \binom{a_{d-1}'-(d-1)}{d-1}_{\!d-1} + \dots + \binom{a_m'-m}{m}_{\!m}.
\end{equation}
By definition, the $(d,d)$-canonical representation of the value on the right-hand side of \eqref{eqn:TopBettiFlagUpperBoundNearCanonicalRep} equals
\[\binom{a_d'-d}{d}_{\!d} + \binom{a_{d-1}'-(d-1)}{d-1}_{\!d-1} + \dots + \binom{a_{m_0}'-m_0}{m_0}_{\!m_0},\]
where $m_0 := \min \{t: m\leq t\leq d, a_t'-t\geq t\}$. In particular, note that if $a'_t - t\geq t$ (for some $m\leq t < d$), then it follows from the defining inequality $a'_{t+1} - \lfloor \tfrac{a'_{t+1}}{t+1}\rfloor \geq a'_t+1$ that $a'_{t+1}-(t+1) \geq t+1$.

Now, since $\beta_{d-1}(\Delta) = a$ by assumption, it then follows from \eqref{eqn:TopBettiFlagCanonicalRep} that
\[\binom{a_d}{d}_{\!d} + \dots + \binom{a_{d-s}}{d-s}_{\!d-s} \leq \binom{a_d'-d}{d}_{\!d} + \dots + \binom{a_{m_0}'-m_0}{m_0}_{\!m_0}.\]
This implies, using Lemma \ref{lem:CanonRep+d}, that
\[\binom{a_d+d}{d}_{\!d} + \dots + \binom{a_{d-s}+d-s}{d-s}_{\!d-s} \leq \binom{a_d'}{d}_{\!d} + \dots + \binom{a_{m_0}'}{m_0}_{\!m_0} \leq f_{d-1}(\Delta),\]
as desired.

To prove the second assertion, suppose that equality holds in \eqref{eqn:fVecFlagLowerBound} for some $i=k$ satisfying $k\geq s+1$. Then by definition, the $(k,d)$-canonical representation of $f_{k-1}(\Delta)$ equals
\[f_{k-1}(\Delta) = \binom{a_d+d}{k}_{\!d} + \binom{a_{d-1}+d-1}{k-1}_{\!d-1} + \dots + \binom{a_{d-s}+d-s}{k-s}_{\!d-s}.\]
Thus by Theorems~\ref{thm:Frohmader} and~\ref{thm:FFK}, and Lemma \ref{lem:iteratedDifferentialOperatorSecond}, we would get
\[f_{i-1}(\Delta) \leq \binom{a_d+d}{i}_{\!d} + \binom{a_{d-1}+d-1}{i-1}_{\!d-1} + \dots + \binom{a_{d-s}+d-s}{i-s}_{\!d-s}\]
for all $i\geq k$. Consequently, the second assertion follows from the first assertion.
\end{proof}

The following balanced analog of Theorem \ref{thm:FlagLowerBoundfVecGivenTopBetti} holds as well.

\begin{corollary}\label{cor:BalancedfVecLowerBound}
Let $\Delta$ be a balanced $(d-1)$-complex with $\beta_{d-1}(\Delta) = a > 0$. If
\begin{equation}\label{eqn:TopBettiBalancedCanonicalRep}
a = \binom{a_d}{d}_{\!d} + \binom{a_{d-1}}{d-1}_{\!d-1} + \dots + \binom{a_{d-s}}{d-s}_{\!d-s}
\end{equation}
is the $(d,d)$-canonical representation of $a$, then
\begin{equation}\label{eqn:fVecBalancedLowerBound}
f_{i-1}(\Delta) \geq \binom{a_d+d}{i}_{\!d} + \binom{a_{d-1}+d-1}{i-1}_{\!d-1} + \dots + \binom{a_{d-s}+d-s}{i-s}_{\!d-s}
\end{equation}
for all $i\geq 0$. Furthermore, if equality holds in \eqref{eqn:fVecBalancedLowerBound} for some $i=k$ satisfying $k\geq s+1$, then equality must hold in \eqref{eqn:fVecBalancedLowerBound} for all $i\geq k$.
\end{corollary}

\begin{proof}
The proof of Theorem \ref{thm:FlagLowerBoundfVecGivenTopBetti} applies almost verbatim. The only two minor differences are that Frohmader's theorem (Theorem \ref{thm:Frohmader}) is not needed, and that the application of Theorem \ref{thm:FlagTopBettiUpperBound} should be replaced by an application of Theorem \ref{thm:RevlexComplexMaximizesBettiNum}.
\end{proof}

\begin{remark}
The inequality in \eqref{eqn:fVecBalancedLowerBound} is tight. This is a consequence of Theorem~\ref{thm:FFK} and Corollary~\ref{cor:NonpureColorShiftedBalancedTopBetti}: For every $a>0$ that satisfies \eqref{eqn:TopBettiBalancedCanonicalRep}, there exists a revlex balanced $(d-1)$-complex such that equality holds in \eqref{eqn:fVecBalancedLowerBound} for all $i\geq 0$.
\end{remark}

\begin{proof}[\normalfont\textbf{Proof of Corollary \ref{cor:FlagLowerBoundfVecTuranGivenTopBetti}.}]
First of all, for any $n\geq d$, Theorem \ref{thm:PureColorShiftedBalancedBetti} gives $\beta_{d-1}(\Delta(T_d(n))) = \binom{n-d}{d}_{\!d}$, so if the $(d,d)$-canonical representation of $a$ is given by \eqref{eqn:TopBettiFlagCanonicalRep}, then $\beta_{d-1}(T)\leq a$ implies $\beta_{d-1}(T) \leq \binom{a_d}{d}_{\!d}$, which forces $f_{i-1}(T)\leq \binom{a_d+d}{i}_{\!d}$ for all $i\geq 0$. Thus by Theorem \ref{thm:FlagLowerBoundfVecGivenTopBetti}, $f_{i-1}(\Delta)\geq f_{i-1}(T)$ for all $i\geq 0$.

Now suppose further that $\beta_{d-1}(T) = a$. This implies that $a = \binom{a_d}{d}_{\!d}$ and $f_{i-1}(T)= \binom{a_d+d}{i}_{\!d}$ for all $i\geq 0$. Consequently, if $f_k(\Delta) = f_k(T)$ for some $k\geq 0$, then Theorem \ref{thm:FlagLowerBoundfVecGivenTopBetti} also gives $f_i(\Delta) = f_i(T)$ for all $i\geq k$. In particular, by Zykov's theorem (Theorem \ref{thm:ZykovThm}), $f_0(\Delta) = f_0(T)$ if and only if $f(\Delta) = f(T)$, if and only if $\Delta \cong T$.
\end{proof}

Theorem \ref{thm:FFK} has the following continuous analog.
\begin{theorem}[{\cite[Thm. 5.1]{FFK1988}}]\label{thm:FFKcontinuous}
Let $1\leq k\leq r$ be integers, and let $\Delta$ be an $r$-colorable complex. If $\alpha \geq 0$ is the unique real number that satisfies $\binom{r}{k}\alpha^k = f_{k-1}(\Delta)$, then $f_{j-1}(\Delta) \geq \binom{r}{j}\alpha^j$ for all $j\in [k]$.
\end{theorem}

\begin{proof}[\normalfont\textbf{Proof of Theorem \ref{thm:FlagLowerBoundfPolyGivenTopBetti}.}]
First of all, the case $a= 0$ is trivially true, since $\dim \Delta = d-1$ implies that $\Delta$ has at least one $(d-1)$-face $F_0$, which yields $f_{k-1}(\Delta) \geq f_{k-1}(\langle F_0\rangle) = \binom{d}{k}$ for all $k\in [d]$, with equality holding for all $k\in [d]$ if and only if $\Delta \cong \Delta(T_d(d))$ is a $(d-1)$-simplex. Henceforth, assume that $a\geq 1$, let $\alpha = \sqrt[d]{a}$, and let
\[a = \binom{a_d}{d}_{\!d} + \binom{a_{d-1}}{d-1}_{\!d-1} + \dots + \binom{a_{d-s}}{d-s}_{\!d-s}\]
be the $(d,d)$-canonical representation of $a$. Suppose that $\Gamma$ is a $d$-colorable complex such that $f_{d-1}(\Gamma) = a$. By Theorem \ref{thm:FFK}, Lemma \ref{lem:iteratedDifferentialOperator}, and Theorem \ref{thm:FFKcontinuous}, we infer that
\begin{equation}\label{eqn:fVecLowerBoundContinuous}
f_{k-1}(\Gamma) \geq \binom{a_d}{k}_{\!d} + \binom{a_{d-1}}{k-1}_{\!d-1} + \dots + \binom{a_{d-s}}{k-s}_{\!d-s} \geq \binom{d}{k}\alpha^k
\end{equation}
for all $0\leq k\leq d$. Consequently, by summing \eqref{eqn:fVecLowerBoundContinuous} over all $0\leq k\leq d$, it then follows from Lemma~\ref{prop:SumTuranCoefficients} that
\[\binom{a_d+d}{d}_{\!d} + \binom{a_{d-1}+d-1}{d-1}_{\!d-1} + \dots + \binom{a_{d-s}+d-s}{d-s}_{\!d-s} \geq \sum_{k=0}^d \binom{d}{k}\alpha^k = (1+\alpha)^d.\]

Now, Theorem \ref{thm:FlagLowerBoundfVecGivenTopBetti} yields
\begin{equation}
f_{d-1}(\Delta) \geq \binom{a_d+d}{d}_{\!d} + \binom{a_{d-1}+d-1}{d-1}_{\!d-1} + \dots + \binom{a_{d-s}+d-s}{d-s}_{\!d-s},
\end{equation}
hence $f_{d-1}(\Delta) \geq (1+\alpha)^d$. Note that Theorem \ref{thm:Frohmader} says $f(\Delta)$ is the $f$-vector of a $d$-colorable complex. Thus by Theorem \ref{thm:FFKcontinuous}, $f_{j-1}(\Delta) \geq \binom{d}{j}(1+\alpha)^j$ for all $j\in [d]$, and the first assertion follows from the binomial theorem.

Finally, if $\alpha$ is an integer, then the Tur\'{a}n complex $\Delta(T_d(d\alpha +d))$ has $(d-1)$-th reduced Betti number $\alpha^d = a$ (say by Theorem \ref{thm:PureColorShiftedBalancedBetti} or K\"{u}nneth's formula for joins), and $f$-polynomial $(1+(\alpha + 1)x)^d$, so the second assertion follows from Zykov's theorem (Theorem \ref{thm:ZykovThm}).
\end{proof}

Analogous to Corollary \ref{cor:BalancedfVecLowerBound}, the following result also follows from the proof of Theorem \ref{thm:FlagLowerBoundfPolyGivenTopBetti}.
\begin{corollary}\label{cor:BalancedfPolyLowerBound}
Let $\Delta$ be a balanced $(d-1)$-complex with $\beta_{d-1}(\Delta) = a \geq 0$. Then the $f$-polynomial of $\Delta$ satisfies
$f_{\Delta}(x) \geq (1 + (\sqrt[d]{a}+1)x)^d$. Furthermore, if $\sqrt[d]{a}$ is an integer, then this inequality is tight, with equality holding if and only if $\Delta \cong \Delta(T_d(d(\sqrt[d]{a}+1)))$.
\end{corollary}

\section{Concluding Remarks}\label{sec:FurtherRemarks}
Given any simplicial complex $\Delta$, its \textit{reduced Betti vector} is $\beta(\Delta) := (\beta_{-1}(\Delta), \beta_0(\Delta), \dots, \beta_{\dim\Delta}(\Delta))$, and its \textit{$(f,\beta)$-vector} is the pair $(f(\Delta), \beta(\Delta))$. Although the $(f,\beta)$-vectors of simplicial complexes are characterized by Bj\"{o}rner--Kalai~\cite{BjornerKalai1988:ExtendedEulerPoincareThm}, there is no known characterization of the possible $(f,\beta)$-vectors for the subfamily of balanced complexes; perhaps the combinatorial techniques by
Duval~\cite{Duval1999:RelativeHomology} and Bj\"{o}rner--Kalai~\cite{BjornerKalai1991:ExtendedEulerPoincareCellComplexes}
are helpful for this subfamily. In view of Theorem~\ref{thm:FrohmaderWithHomology}, we raise the following questions:
\begin{problem}
Is there a flag complex $\Delta$ such that $(f(\Delta),\beta(\Delta))$ is not the the $(f,\beta)$-vector of any balanced complex? Perhaps this is true even when restricted to $(f(\Delta),\beta_{\dim \Delta}(\Delta))$?
\end{problem}

In view of Theorem~\ref{thm:FlagLowerBoundfVecGivenTopBetti}, the following conjectures provide quantitative refinements of Meshulam's theorem~\cite[Thm. 1.1]{Meshulam2003:DominationHomology} in \emph{all} dimensions.
\begin{conjecture}\label{conj:6.2}
Let $\Delta$ be a flag complex with $\beta_{k-1}(\Delta) = a > 0$. If
\begin{equation*}
a = \binom{a_k}{k}_{\!k} + \binom{a_{k-1}}{k-1}_{\!k-1} + \dots + \binom{a_{k-s}}{k-s}_{\!k-s}
\end{equation*}
is the $(k,k)$-canonical representation of $a$, then
\begin{equation*}
f_{i-1}(\Delta) \geq \binom{a_k+k}{i}_{\!k} + \binom{a_{k-1}+k-1}{i-1}_{\!k-1} + \dots + \binom{a_{k-s}+k-s}{i-s}_{\!k-s}
\end{equation*}
for all $i\geq 0$.
\end{conjecture}
Conjecture~\ref{conj:6.2}, if true, would imply the following conjecture, by following the proof of Corollary~\ref{cor:FlagLowerBoundfVecTuranGivenTopBetti} with the obvious changes in notation.
\begin{conjecture}
Let $\Delta$ be a flag complex with $\beta_k(\Delta) = a > 0$. Let $T$ be any $k$-dimensional Tur\'{a}n complex that satisfies $\beta_k(T) \leq a$.
Then $f(\Delta)\geq f(T)$ componentwise.
\end{conjecture}

\bibliographystyle{plain}
\bibliography{References}

\end{document}